\documentclass[11pt]{amsart}

\usepackage{amscd,amssymb,amsmath,graphicx,mathrsfs,color}
\usepackage[all]{xy}

\usepackage[dvips]{hyperref}
\usepackage[TS1,OT1,T1]{fontenc}

\newtheorem{theorem}{Theorem}[section]
\newtheorem{lemma}[theorem]{Lemma}
\newtheorem{corollary}[theorem]{Corollary}
\newtheorem{proposition}[theorem]{Proposition}

\theoremstyle{definition}
\newtheorem{definition}[theorem]{Definition}

\theoremstyle{remark}
\newtheorem{remark}[theorem]{Remark}

\numberwithin{equation}{section}

\newcommand{\R}{\mathbb R}

\renewcommand{\O}{\mathcal{O}}

\newcommand{\la}{\langle}
\newcommand{\ra}{\rangle}
\newcommand{\ca}{\mathcal}
\newcommand{\scr}{\mathscr}

\renewcommand{\L}{\textbf{L}}

\newcommand{\lotimes}{\stackrel{\L}{\otimes}}

\newcommand{\ep}{\epsilon}

\title{Deformations of vector bundles on coisotropic subvarieties via the Atiyah class}
\author{Jeremy Pecharich}
\address{UC-Irvine, Mathematics Department, Irvine, CA}
\email{jpechari@math.uci.edu}
\label{l1}\label{l2}\label{l3}

\date{}

\begin{document}

\maketitle

\begin{abstract}
Using the Atiyah class we give a criterion for a vector bundle on a coisotropic subvariety, $Y$, of an algebraic Poisson variety $X$ to admit a first and second order noncommutative deformation.  We also show noncommutative deformations of a vector bundle are governed by a curved dg Lie algebra which reduces to the classical relative Hochschild complex when the Poisson structure on $X$ is trivial.
\end{abstract}

\section{Introduction}

Let $X$ be a smooth algebraic variety and let $\O_X$ denote the sheaf of regular functions on $X$.  Recall a deformation quantization of order $2$ of $X$ is a flat sheaf $\ca A_2$ of algebras over $k[\epsilon]/\epsilon^3$ such that on affine subsets $U_i$ of a Zariski open covering of $X$ we have $\ca A_2|_{U_i}\simeq (\O_X\oplus \epsilon \O_X \oplus \epsilon^2 \O_X)|_{U_i}$ as sheaves of $k[\epsilon]/\epsilon^3$-modules.  The product is locally given by 
\[a\ast_ib=ab+\epsilon \alpha_1^{Xi}(a,b)+\epsilon^2\alpha_2^{Xi}(a,b)\]
where $a,b$ are local section of $\O_{U_i}$ and $\alpha_1^{Xi}(a,b)=\frac{1}{2}P(da,db)$ for a globally defined bivector $P\in H^0(X,\wedge^2T_X)$ and $\alpha_2^{Xi}$ is a bidifferential operator.  On double intersection $U_i\cap U_j$ the restrictions are identified by sending a regular function $f$ to $f+\epsilon \beta_1^{Xij}(f)+\epsilon^2\beta_2^{Xij}(f)$ where $\beta_1^{Xij},\beta_2^{Xij}$ are differential operators from $\O_{U_i\cap U_j}$ to $\O_{U_i\cap U_j}$.  In a far fancier language than we will need here $U\mapsto \ca A_2(U)$ is a presheaf of algebroids \cite{Ko}.

In this paper we consider the higher rank version of \cite{BGP}:  let $Y\subset X$ be a smooth closed coisotropic subvariety of a smooth Poisson variety $X$ and $E$ a vector bundle on $Y$.  Viewing $E$ as coherent $\O_X$-module a natural question to ask is when does $E$ admit a flat second order deformation to an $\ca A_2$-module.  This means, we want a coherent sheaf $\ca E_2$ which splits locally on an affine open cover $\{U_i\}$, with a module action given by 

\begin{equation*}
a\ast_ie=ae+\epsilon \alpha_1^i(a,e)+\epsilon^2\alpha_2^i(a,e)
\end{equation*}
and transition functions on $U_i\cap U_j$ given by 
\[e\mapsto e+\epsilon \beta_1^{ij}(e)+\epsilon^2\beta_2^{ij}(e)\]
where $a,e$ are local sections of $\O_X,$ and $E$, respectively, and $\alpha_1^i,\alpha_2^i$ and $\beta_1^{ij},\beta_2^{ij}$ are (bi)differential operators.

For simplicity we set $F(E):=F\otimes_{\O_Y}End_{\O_Y}(E)$ where $F$ is a coherent sheaf on $X$. Using spectral sequences there are three obstructions to the existence of $\ca E_1$ in $H^0(Y,\wedge^2N(E))$, $H^1(Y,N(E))$ and $H^2(Y,\O_Y(E))$ where $N$ is the normal bundle of $Y$ in $X$.  The first obstruction measures whether $\ca E_1$ exists locally in the Zarkiski/\'etale topology.  If an infinitesimal deformation exists locally the class in $H^1(Y,N(E))$ is well-defined and its vanishing is equivalent to the existence of transition functions $\beta_1^{ij}$ which agree with the module structure.  The class in $H^2(Y,\O_Y(E))$ is well-defined when the previous class vanishes and this class vanishes precisely when the transition functions satisfy the cocycle condition on each triple intersection $U_i\cap U_j\cap U_k$.  The class in $H^2(Y,\O_Y(E))$ for $Y=X$ has been studied in \cite{BBP}.  When $Y\subset X$ the author believes it is connected to Rozanksy-Witten invariants but we leave this for future study cf. \cite{BGNT}. 

In \cite{BG} it was shown the first obstruction class is the image of $P$ in $H^0(Y,\wedge^2N(E))$ and its vanishing is equivalent to $Y$ be coisotropic cf. Lemma \ref{coisotropic}.  Recall, that $Y$ is coisotropic if $P(I_Y,I_Y)\subset I_Y$ where $I_Y$ is the sheaf of regular functions vanishing on $Y$.  With this in mind, we now assume that $Y$ is \emph{coisotropic}.  By coisotropness of $Y$ the bivector $P$ defines a morphism $p:N^\vee\to T_Y$ along with its adjoint $p^*:\Omega_Y^1\to N$.  Throughout the paper we fix a line bundle $L$ which admits a first/second order deformation.  In the case when $\beta_1^X\equiv 0$ and $\alpha_2^X$ is symmetric we can take $L=(\det N)^{1/2}$ \cite{BGP}.    

Denote by 
\[at_{N}(E\otimes_{\O_Y} L^\vee):=p^*(at(E\otimes_{\O_Y} L^\vee))\]
the Yoneda product of $p^*$ and $at(E\otimes L^\vee)\in H^1(Y,\Omega^1_Y(E))$.  Where $at(M)$ is the Atiyah class of a vector bundle $M$ \cite{Atiyah}.

\begin{theorem}
\label{firstorder}
Let $X$ be a smooth algebraic variety with a bivector $P$ and $Y$ a smooth coisotropic subvariety with a vector bundle $E$.  If $E$ admits a first order deformation $\ca E_1$ then 
\begin{equation*}
at_N(E\otimes_{\O_Y} L^\vee)=0
\end{equation*}
in $H^1(Y,N(E))$.  If, in addition, $H^2(Y,\O_Y(E))=0$ the above equality in $H^1(Y,N(E))$ is also sufficient for the existence of a first order deformation.  In particular, a first order deformation exists when $X$ and $Y$ are affine.  Moreover, in the affine case there is a globally split deformation, i.e. $\ca E_1\simeq E\oplus \ep E$ as sheaves of $k[\ep]/\ep^2$-modules.
\end{theorem}

A first order deformation up to isomorphism is given by a collection of operators $\gamma_E^i:N^\vee\to \ca D^1_\heartsuit(E)$ on a Zariski open cover $\{U_i\}$ which satisfy a gluing condition on $U_i\cap U_j$ cf. Proposition \ref{isoclass};  $\ca D^1_\heartsuit(E)$ are first order differential operators with scalar principal symbol.  Using this we give an explicit connection on $E\otimes L^\vee$ which represents the class in Theorem 1.1.

With regards to second order deformations we tacitly assume $\beta_1^X\equiv 0$.  For $\ca A_2$ a second order deformation $\{a,b\}_P:=2\alpha_1^X(a,b)$ is a Lie bracket.  By coisotropness of $Y$ the conormal bundle $N^\vee=I_Y/I_Y^2$ becomes a Lie algebra where $I_Y$ is the ideal of functions that vanish on $Y$.  By Proposition \ref{isoclass} a first order deformation gives a global operator $\gamma$ which defines a morphism between Lie algebras which will not respect the bracket in general (we are using the assumption that $\beta_1^X\equiv 0$). The \emph{curvature}, $c(\gamma)$, measures the failure of $\gamma$ to be a morphism of Lie algberas. We define the \emph{normal complex of $E$} as 
\begin{equation}
\ca N^\bullet_{E}:\left\{0\to \O_Y(E)\stackrel{d_{N_E}}{\to} N(E)\stackrel{d_{N_E}}{\to} \wedge^2N(E)\stackrel{d_{N_E}}{\to}\cdots \right\}
\end{equation}
where the odd derivation is given by 
\begin{align*} 
d_{N_E}\omega(x_0,\ldots,x_{n+1})&=\sum_{j=0}^{n+1}(-1)^j[\gamma(x_j,\cdot),\omega(x_0,\ldots,\widehat{x}_j,\ldots,x_{n+1})]\\
&+\sum_{i<j}(-1)^{i+j}\omega(\{x_i,x_j\}_P,x_0,\ldots,\widehat{x}_i,\ldots,\widehat{x}_j,\ldots,x_{n+1})
\end{align*}
and the $x_i$'s are local sections of $N^{\vee}$.  A straightforward but tedious calculation shows $c(\gamma)\in H^0(Y,\wedge^2N(E))$.  Another standard computation using the Jacobi identity for $\{,\}_P$ gives $d_{N_E}^2\omega=[c(\gamma),\omega]$ and $d_{N_E}c(\gamma)=0$ (second Bianchi identity) which make $\ca N_{E}$ into a curved dg Lie algebra \cite{Positselski2}. When $\ca E_2$ exists $\ca N_{E}$ is \emph{weakly obstructed} meaning $[c(\gamma),\omega]=0$ for all $\omega$.  In the case when $E$ is rank $1$ the complex is automatically ``weakly obstructed.''

 For ease of the notation we make a definition similar to Deligne's \emph{$\lambda$-connections}.
\begin{definition}
A \emph{$(\lambda,\mu)$-connection} on a vector bundle $E$ is a $k$-linear operator $\nabla:N^\vee\to \ca D^1_{\heartsuit}(E)$ whose principal symbols are 
\begin{equation*}
\nabla(ax,e)-a\nabla(x,e)=\lambda p(x)(a)e;\qquad \nabla(x,ae)-a\nabla(x,e)=\mu p(x)(a)e
\end{equation*}
We denote the set of $(\lambda,\mu)$-connections by $\ca D^1_{(\lambda,\mu)}(E)$. A $(0,1)$-connection is an $N^{\vee}$-connection from \cite{BGP}.  

\end{definition}

\begin{theorem}
\label{secondorder}
Assume $E$ admits a second order deformation $\ca E_2$ and $\beta_1^X\equiv 0$ then $c(\nabla)=0$ where $\nabla$ is the $(0,1)$-connection on $E\otimes L^\vee$ given by the first order deformation.  If $H^1(Y,N(E))=H^2(Y,\O_Y(E))=0$ this equality also implies the existence of a second order deformation.  For affine $X$ and $Y$ the deformation of $E$ may be chosen globally split: $\ca E_2\simeq E\oplus \ep E\oplus \ep^2 E$ as sheaves of $k[\ep]/\ep^3$-modules.
\end{theorem}

The outline of the paper is as follows: In section 2 we give the proofs of Theorems 1.1 and 1.2.  Section 3 we show a flat bundle can be deformed to second order after twisting by a line bundle which admits a second order deformation.  The last section shows deforming a module is not governed by dg Lie algebra but a curved dg Lie algebra defined over ring of formal power series.  We have also included an appendix with the module equations to order 2 and a version of the HKR theorem which we use throughout the paper.

\bigskip
\noindent
\emph{Acknowledgements:}  I would like to thank my advisor, Vladimir Baranovsky, for the constant encouragement and countless hours of discussion.

\section{First and second order deformations}
\subsection{First Order}
Suppose there is a first order deformation $\ca E_1$ of a vector bundle $E$.  This means there is an affine open cover $\{U_i\}$ of $X$ such that $\ca E_1\simeq (E\oplus \epsilon E)|_{U_i}$ for all $i$ as sheaves of $k[\epsilon]/\epsilon^2$-modules.  In particular by \eqref{alpha-one} the operator $\alpha_1^i$ vanishes on $I_Y^2\otimes E$.  Denote by $\gamma^i_E$ the restriction of $\alpha^i_1$ to $N^\vee\otimes E\to E$ which we view as a bidifferential operator on $Y\cap U_i$. Applying \eqref{alpha-one} twice implies $\gamma_E^i$ is a $(1/2,1)$-connection. On double intersections \eqref{beta-one} reduces to 

\begin{equation}
\label{doubleintersection}
\gamma_E^j(x,e)-\gamma_E^i(x,e)+\beta_1^{Xij}(x)e=0
\end{equation}

\noindent
The following proposition proven in \cite{BGP} determines a first order deformation up to isomorphism.

\begin{proposition}
\label{isoclass}
The collection of $(1/2,1)$-connections $\{\gamma_E^i\}$ defines $\ca E_1$ uniquely up to isomorphism if $H^1(Y,\O_Y(E))=0$. Conversely, if $Y$ satisfies $H^2(Y,\O_Y(E))=0$, for any such collection of $(1/2,1)$-connections satisfying \eqref{doubleintersection}, there exists a first order deformation $\ca E_1$ inducing it.
\end{proposition}

A local first order deformation exists if and only if $P$ projects to zero in $H^0(Y,\wedge^2N(E))$ which is a priori weaker than being coisotropic.  However, an easy application of the HKR theorem shows a first order deformation locally exists if and only if $Y$ is coisotropic.

\begin{lemma}
\label{coisotropic} The projection of $P$ is contained in $H^0(Y,\wedge^2N)\subset H^0(Y,\wedge^2N(E))$.  Locally there exists a first order deformation if and only if the projection of $P$ vanishes, that is $Y$ is coisotropic.
\end{lemma}

\begin{proof}
The following proof is an application Lemma \ref{HKR}, which will be used repeatedly, so we will give all the details.  A first order deformation locally exists if and only if there is an $\alpha_1^i$ which satisfies \eqref{alpha-one}.  Applying Lemma \ref{HKR} we must show that $G(a,b,e):=\alpha_1^X(a,b)e$ is symmetric when restricted to $I\otimes I\otimes E\to E$ and is a cocycle i.e. 
$$
aG(b,c,e)-G(ab,c,e)+G(a,bc,e)-G(a,b,ce)=0
$$
This holds since $\alpha_1^X$ is a cocycle in $C^*(\O_X,\O_X)$ \cite{Kodef}.  The projection of $P$ in $H^0(Y,\wedge^2N(E))$ is the anti-symmetrization of $G(x,y,e)$ for $x,y\in I_Y$. This is a scalar endomorphism since $G(a,b,e)$ is a scalar endomorphism.  The cocycle $\alpha_1^X$ comes from the Poisson structure hence is antisymmetric. Since a cocycle is symmetric if and only if the anti-symmetrization vanishes we must have $2\alpha_1^X(x,y)e=0$ for all $x,y\in I$ and $e\in E$.  This implies $\alpha_1^X(x,y)\in I$ for all $x,y\in I$ which is the coisotropic condition.    
\end{proof}

\noindent
We first need a couple of lemmas which will be useful in the proof of Theorem \ref{firstorder} and later in the text.

\begin{lemma} 
\label{tensordeformation}
Let $E,F$ be two vector bundles on $Y$ with connections $\gamma_E\in D_{(\lambda,\mu)}(E)$ and $\gamma_F\in \ca D_{(\lambda',\mu)}(F)$.  Then 
\begin{equation}
\label{Leibniz}
\gamma_{E\otimes F}(x,e\otimes f):=\gamma_E(x,e)\otimes f+e\otimes \gamma_F(x,f)
\end{equation}
defines an element of $\ca D_{(\lambda+\lambda',\mu)}(E\otimes_{\O_Y}F)$.  The curvature of $\gamma_{E\otimes F}$ is given by 
$$
c(\gamma_{E\otimes F})(x,y)(e\otimes f)=c(\gamma_E)(x,y)(e)\otimes f+e\otimes c(\gamma_F)(x,y)(f)
$$
\end{lemma}

\begin{proof}  The proof is by direct calculation: let $a\in \O_Y$, $x\in N^\vee$ then 
\begin{align*}
\gamma_{E\otimes F}(ax,e\otimes f)-a\gamma_{E\otimes F}(x,e\otimes f)&=\gamma_E(ax,e)\otimes f+e\otimes \gamma_F(ax,f)-a\gamma_E(x,e)\otimes f-ae\otimes \gamma_F(x,f)\\
&=\lambda p(x)(a)e\otimes f+\lambda'p(x)(a)e\otimes f\\
&=(\lambda+\lambda')p(x)(a)e\otimes f
\end{align*}
The other symbol gives 
\begin{align*}
\gamma_{E\otimes F}(x,ae\otimes f)-a\gamma_{E\otimes F}(x,e\otimes f)&=\gamma(x,ae)\otimes f+ae\otimes \gamma_F(x,f)-a\gamma_E(x,e)\otimes f-ae\otimes \gamma_F(x,f)\\
&=\mu p(x)(a)e\otimes f
\end{align*}
The curvature formula is standard from differential geometry.
\end{proof}

\begin{lemma}
Let $L$ be as above then there exists a collection $\{\gamma^i_{L^\vee}\}$ of $(-1/2,1)$-connections on $L^{\vee}$ such that on double intersections we have $\gamma_{L^\vee}^i(x,l^\vee)-\gamma_{L^\vee}^j(x,l^\vee)-\beta_1^{Xij}(x)l^\vee=0$.
\end{lemma}
\begin{proof}
Fix a section $l$ of $L$ and let $l^\vee$ be the dual section of $L^\vee$ under the non-degenerate $\O_Y$-bilinear pairing $\la\cdot,\cdot\ra:L\otimes_{\O_Y}L^\vee\to \O_Y$.  Define an operator on $L^\vee$ via the Leibniz rule 
\begin{equation}
\label{Lconnection}
\partial_x(\la l,l^\vee\ra)=\gamma_L(x,l)\otimes l^\vee+l\otimes\gamma_{L^{\vee}}(x,l^\vee)
\end{equation}
By Lemma \ref{tensordeformation}, $\gamma_{L^{\vee}}$ is a $(-1/2,1)$-connection.  The formula on double intersections holds since the left hand side is a global connection.    
\end{proof}

\begin{proof}[Proof of theorem \ref{firstorder}]  Suppose $\gamma_E^i$ exists and on $U_i$ we define $\nabla^i:N^\vee\to \ca D_\heartsuit^1(E\otimes L^\vee)$ by 
$$
\nabla^i(x,e\otimes l^\vee)=\gamma_E^i(x,e)\otimes l^\vee+e\otimes \gamma_{L^\vee}^i(x,l^\vee)
$$
which is $(0,1)$-connection by the previous two lemmas.  It is also easy to check on double intersections that $\nabla^i-\nabla^j=0$.  The cocycle $\nabla^i-\nabla^j$ represents the class $at_N(E\otimes L^\vee)\in H^1(Y,N(E))$.  Since the connections $\nabla^i$ are chosen up to a section of $H^0(U_i,N(E))$ we see that 
$$
at_N(E\otimes_{\O_Y} L^\vee)=0
$$

Conversely, if the equality holds, we can find $(0,1)$-connections $\nabla^i$ on $U_i$ which glue to a global connection.  We now define $\gamma_E^i$ to be
\[\gamma_E^i(x,e)=\frac{\nabla^i(x,e\otimes l^\vee)-e\otimes \gamma^i_{L^\vee}(x,l^\vee)}{l^\vee}\]
where $l^\vee$ is any local section of $L^\vee$. We now apply the previous proposition.  \end{proof}

\noindent
\textit{Remark.}  In the case when $H^2(Y,\O_Y(E))=0$ there is a bijection
$$
\left\{ \text{first order deformations of $E$, $\ca E_1$}    \right\}/\sim   \leftrightarrow \left\{\text{$(0,1)$-connections on $E\otimes_{\O_Y} L^\vee$}\right\}
$$
To any $(0,1)$-connection on $E\otimes_{\O_Y} L^\vee$ there is a collection $(1/2,1)$-connection on $E$ by Lemma \ref{tensordeformation} which satisfy \eqref{doubleintersection}.  Applying Lemma \ref{isoclass} shows there is a bijection.  There is a $(0,1)$-connection when $at_N(E\otimes_{\O_Y} L^\vee)=0$ in $H^1(Y,N(E))$.

\begin{proposition}
\label{autofirst}
Given a first order deformation $\ca E_1$, its group of automorphisms restricting to the identity modulo $\epsilon$ is isomorphic to $H^0(Y,\O_Y(E))$.  If $H^1(Y,\O_Y(E))=H^2(Y,\O_Y(E))=0$ and the condition imposed on $at(E\otimes_{\O_Y} L^\vee)$ holds, then the set of isomorphism classes of first order deformations is a torsor over $H^0(Y,N(E))$.
\end{proposition}

\begin{proof} 
Let $\phi:\ca E_1\to \ca E_1$ be an automorphism which restricts to the identity modulo $\epsilon$ then $1-\phi$ takes values in $\epsilon \ca E_1$ and hence descends to $\ca E_1/\epsilon \ca E_1\simeq E\to E\simeq \epsilon E$. The map $1-\phi$ gives a section $\phi_1\in H^0(Y,\O_Y(E))$ since it is $\O_Y$-linear.  Therefore $\phi=1+\epsilon \phi_1$. 

By Proposition \ref{isoclass} the vanishing of the cohomology groups implies the isomorphism class is uniquely determined by the choice of $\gamma_i$.  The difference of two $(1/2,1)$-connections will be an $\O_Y$-bilinear map $N^\vee\times E\to E$.  This means the difference will be a section of $N(E)$ over $U_i$.  Moreover, equation \eqref{doubleintersection} shows that such a section will glue on $U_i\cap U_j$.  
\end{proof}

\subsection{Second Order }
In this subsection assume that $\beta_1^X\equiv 0$.  When $\ca A_2$ exists locally there are bidifferential operators $\alpha_2^{Xi}$ which satisfy \eqref{alphaX-two} along with gluing conditions on double intersections \eqref{betaX-two}.  By skew-symmetry of $\alpha_1^X$ and \eqref{alphaX-two} the anti-symmetrization $\scr A_2^i(a,b):=\alpha_2^{Xi}(a,b)-\alpha_2^{Xi}(b,a)$ satisfies 
$$
a\scr A_2^i(b,c)-\scr A_2^i(ab,c)+\scr A_2^i(a,bc)-\scr A_2^i(a,b)c=0
$$
The HKR isomorphism shows $\scr A_2^i$ is given by a bivector in $H^0(U_i,\wedge^2T_X)$.  Since, $\beta_1^X\equiv 0$ the collection $\{\scr A_2^i\}$ glues to a global bivector $\scr A_2\in H^0(X,\wedge^2T_X)$.

\begin{proof}[Proof of theorem \ref{secondorder}] Recall, we now assume $\beta_1^X\equiv 0$. Let $X$ be affine and suppose we are given a first order deformation $\ca E_1\simeq E\oplus \ep E$ from the previous theorem.  To extend to $\ca E_2$ we need to find $\alpha_2$ which solves \eqref{alpha-two}.  By Lemma \ref{HKR} the existence of $\alpha_2$ is equivalent to the vanishing of the antisymmetrization of \eqref{alpha-two} when restricted to $I_Y$ i.e. we must solve 
\[\scr A_2(x,y)e=\gamma(x,\gamma(y,e))-\gamma(y,\gamma(x,e))-\gamma(\{x,y\}_P,e)\] 

By assumption, $L$ has a second order deformation which implies $c(\gamma_L)(x,y)(l)=\scr{A}_2(x,y)l$ by the previous paragraph.  The left hand side of \eqref{Lconnection} is a flat connection therefore $c(\gamma_{L^\vee})(x,y)(l^\vee)=-\scr{A}_2(x,y)l^\vee$.  Using Lemma \ref{tensordeformation} we see $c(\nabla)=0$.  

In general, the same reasoning implies the existence of operators $\alpha_2^i(a,e)$ on affine subsets $U_i$.  By Lemma \ref{HKR} the existence of $\beta_2^{ij}$ satisfying \eqref{beta-two} is equivalent to 
\begin{equation}
\label{secondclass}
\alpha_2^j(x,e)-\alpha_2^i(x,e)+\beta_2^{Xij}(x)e+\alpha_1^j(x,\beta_1^{ij}(e))-\beta_1^{ij}(\alpha_1^i(x,e))=0
\end{equation}
A straightforward calculation shows the RHS of \eqref{beta-two}  is a Hochschild cocycle and hence $\O_X$-bilinear when restricted to $I\otimes E$ by lemma \ref{anticocycle}.  Moreover, it vanishes for $x\in I_Y^2$ and therefore descends to $N^\vee$ defining a section of $c_{ij}\in H^0(U_i\cap U_j,N(E))$.  If this class vanishes in $H^1(Y,N(E))$ then there are sections $c_i\in H^0(U_i,N(E))$ such that $c_{ij}=c_i-c_j$ on $U_i\cap U_j$.  For fixed $i$, the conormal sequence
$$
0\to N^\vee\to \Omega_X^1|_Y\to \Omega_Y^1\to 0
$$
splits since $U_i\cap Y$ is affine and the three sheaves in the sequence are locally free.  Denote the surjection by $\pi_i:\Omega^1_X|_Y\to N^\vee$.  The expression, $c_i(x)e$ can then be lifted to an operator $\psi_i(a,e)=c_i(\pi_i(da))e$ which is a Hochschild cocycle.  We now replace $\alpha_i(a,e)$ with $\alpha_i(a,e)-\psi_i(a,e)$ to ensure \eqref{beta-two} holds.  Since $H^2(Y,\O_Y(E))=0$ we can adjust $\beta_2^{ij}$ by adding $\O_Y$ linear operators $E\to E$ on $U_i\cap U_j$ so the cocycle condition for gluing function holds on $U_i\cap U_j\cap U_k$.  This completes the proof. 
\end{proof}

\begin{remark}
If $X$ is affine and $\alpha_2^X$ is symmetric the two previous theorems imply any vector bundle supported on a coisotropic subvariety along with a flat connection along the \emph{null foliation} has a second order quantization.  When the subvariety, $Y$, is Lagrangian any vector bundle on $Y$ with a flat $(0,1)$-connection has a second order quantization. \end{remark}

\begin{proposition}
Assume $H^1(Y,\O_Y(E))=0$. (a) Let $E$ be a vector bundle which admits a second order deformation $\ca E_2$ and let $\phi:\ca E_1\to \ca E_1$ be an automorphism restricting to the identity modulo $\epsilon$.  If $\phi_1\in H^0(Y, \O_Y(E))$ is the regular section corresponding to $\phi$ via Proposition \ref{autofirst} then $\phi$ extends to a second order automorphism $\phi_2:\ca E_2\to \ca E_2$ if and only if $d_{N_{E}}\phi_1=0$.  In this case the set of all extensions $\phi_2$ is a torsor over $H^0(Y,\O_Y(E))$.

(b)  Assume that $E$ has two second order deformations $\ca E_2$ and $\ca E_2'$ such that for their first order truncations we have 
\[\ca E_1'=\ca E_1+\zeta; \qquad \zeta\in H^0(Y,N(E))\]
in the sense of the torsor structure of Proposition \ref{autofirst}. Then $d_{N_{E}}\zeta+[\zeta,\zeta]=0$ in $H^0(Y,\wedge^2N(E))$. If $H^2(Y,\O_Y(E))=H^1(Y,N(E))=0$ and $\zeta$ satisfying $d_{N_E}\zeta+[\zeta,\zeta]=0$ is fixed, then for a given $\ca E_2$ and $\ca E_1'$ the set of isomorphism classes of $\ca E_2'$ is a torsor over $H^0(Y,N(E))$.
\end{proposition}

\begin{proof}
Locally the automorphism is of the form $e\mapsto e+\ep\phi_1(e)+\ep^2\eta(e)$.  If we write out the equation $\phi(a\star e)=a\star\phi(e)$ then we see that 
$$
\eta(ae)-a\eta(e)=\alpha_1(a,\phi_1(e))-\phi_1(\alpha_1(a,e))
$$
By Lemma \ref{HKR} such an $\eta$ exists if and only if the RHS vanishes for $a\in I_Y$.  This is precisely the condition $d_{N_E}\phi_1=0$.

We now show that two open sets $U_i$ and $U_j$ are related by the transition $e\mapsto e+\ep\beta_1^{ij}(e)+\ep^2\beta_2^{ij}(e)$.  This leads to 
$$
\eta_i(e)-\eta_j(e)-\beta_1^{ij}(\phi_1(e))+\phi_1(\beta_1^{ij}(e))=0
$$ 
which may not hold with the original $\eta_i$, $\eta_j$ but these may be adjusted by an $\O_Y$-linear endomorphism of $E$ on $U_i$ and $U_j$.  The defining equations for $\eta$, equation \eqref{beta-one}, and that $\phi_1$ is an $\O_Y$-linear endomorphism imply the LHS is $\O_Y$-linear and defines a cocycle in $H^1(Y,\O_Y(E))$.  Since $H^1(Y,\O_Y(E))=0$ the $\eta_i$'s can adjusted to ensure the local automorphisms agree on double intersections.  The only remaining ambiguity for $\eta_i$ is the addition of a globally defined $\O_Y$-linear endomorphism of $E$.

To prove (b) we recall that if $H^1(Y,N(E))=0$ then a first order deformation is determined by a collection of $(1/2,1)$-connections.  Given two second order deformations $\ca E_2$ and $\ca E_2'$ whose first order deformations satisfy
$$
\gamma_E^i(x,e)-\gamma_E^{i'}(x,e)=\zeta(x)e
$$
for some $\zeta\in H^0(Y,N(E))$ implies $c(\gamma)=c(\gamma'+\zeta)$. A quick calculation shows $c(\gamma'+\zeta)=c(\gamma')+d_{\ca N_E}\zeta+[\zeta,\zeta]$.  Since $\gamma$ and $\gamma'$ extend to second order theorem \ref{secondorder} shows their curvatures are equal $c(\gamma)=c(\gamma')$.

Conversely, if $d_{\ca N_E}\zeta+[\zeta,\zeta]=0$ then the above calculation shows that $\gamma_E^i(x,e)+\zeta(x)e$ satisfies the curvature equation of Theorem \ref{secondorder}.  Hence there exists local operators $\alpha_2^i(x,e)$ satisfying \eqref{alpha-two}.  The assumptions $H^2(Y,\O_Y(E))=H^1(Y,N(E))=0$ imply that all obstructions to existence of $\beta_2^{ij}$ which satisfy \eqref{beta-two} vanish. The second order deformation corresponding to $\zeta$ can be found. 
\end{proof}

\section{Deforming flat vector bundles}
Throughout this section we will assume that $X,Y$ are affine varieties.  The statements can be generalized to the non-affine case with suitable cohomology vanishing which we leave to the motivated reader. Furthermore, we also assume that $P$ is non-degernate i.e. symplectic.  In this case $p:N^\vee\to T_Y$ is an embedding of vector bundles.  The image, $T_F$, is the \emph{null-foliation} of $Y$.   

\subsection{Flat vector bundles} The main result of this section is the following theorem:

\begin{theorem}
\label{Dmodule}
There is a bijection
\begin{equation*}
\xymatrix{\ca M_F(Y)\ar@/^/[r]^{quant}&\ar@/^/[l]^{dequant} \ca Q_2(Y)}
\end{equation*}
where $\ca M_F(Y)$ is the set of vector bundles on $Y$ which admit a $(0,1)$-connection flat along the null foliation and $\ca C_2(Y)$ is the set of vector bundles on $Y$ which admit a second order deformation.
\end{theorem}

\begin{proof}
For $M$ a vector bundle with a connection $\nabla_M$ that is flat along $T_F$ let $quant(M):=M\otimes_{\O_Y}L$. Define a $(1/2,1)$-connection on $M\otimes_{\O_Y}L$ via \eqref{Leibniz}.  Therefore by Proposition \ref{isoclass} $M\otimes_{\O_Y} L$ admits a first order deformation.  By Lemma \ref{tensordeformation} $c(\gamma_M)=\scr{A}_2$ then Theorem \ref{secondorder} shows $M\otimes L$ admits a second order deformation.

If $M$ is a bundle which admits a second order deformation then 
$$
\nabla_{M\otimes L^\vee}(x,m\otimes l^\vee)=\gamma_M(x,m)\otimes l^\vee+m\otimes \gamma_{L^\vee}(x,l^\vee)
$$   
is a flat $(0,1)$-connection on $dequant(M):=M\otimes_{\O_Y} L^\vee$ again by Lemma \ref{tensordeformation}.
\end{proof}

\begin{remark}
Let $X$ be a smooth variety then to any $\ca D$-module one can associate a coisotropic subvariety $Y\subset T_X$ i.e. the singular support.  Let $W\subset X$ be a smooth subvariety with a local system which we view as a coherent $\ca D$-module on $X$ via the direct image.  The singular support is then $N^*_{X/W}\subset T^*_X$ which is a Lagrangian subvariety with a local system induced by the local system on $W$.  Denote by $\pi:T^*_X\to X$ the projection map. Using the sequence 
$$
0\to \pi^*T^*_W\to N_{T^*_X/W}\to \pi^*N_{X/W}\to 0
$$
we see that $\wedge^*N_{T^*_X/W}$ has a second order deformation.  By the above theorem, the local system on $N_{X/W}$ can be deformed to second order over the deformation quantization of $\O_{T^*X}$ given by the standard symplectic form on $T_X$ after twisting by $\wedge^*N_{T^*_X/W}$.

In unpublished work, Dmitry Kaledin has proven the same theorem for smooth $\ca D$-modules with smooth support but for infinite order deformations using completely different methods. \cite{Kaledin}.
\end{remark}

\noindent
A direct corollary of the above proof shows that $\ca Q_i(Y)$ is a symmetric monoidal category

\begin{corollary} If $Y\subset X$ are affine then the category of second order deformations, $\ca Q_2(Y)$ is a symmetric monoidal category.    \end{corollary}

\begin{proof}

Define a tensor product via
\begin{align*}
\boxtimes:\ca Q_1(Y)\times \ca Q_1(Y)&\to \ca Q_1(Y)\\
(E_1,E_2)&\to E_1\otimes_{\O_Y}E_2\otimes_{\O_Y}L^\vee
\end{align*}
The identity element is given by $L$.  It is then clear $E_1\boxtimes E_2$ admits a first/second order deformation with the above hypotheses.  Moreover, it is easy to check that $\boxtimes$ is symmetric and associative using Lemma \ref{tensordeformation}.
\end{proof}

\noindent
In the case when $Y$ is lagrangian i.e. $\dim Y=1/2\dim X$, there is an isomorphism $p:N^\vee\simeq T_Y$.  To any vector bundle with a flat connection there corresponds module over $\ca A_2$ which splits as sheaf of $k[\ep]/\ep^3$-modules.  When $\alpha_2^X$ is symmetric we can take $L=(\det N)^{1/2}=K_Y^{1/2}$ if it exists.  The quantization map is given by twisting by $K_Y^{1/2}$.

\subsection{Atiyah algebra}
Define the \emph{null foliation Atiyah algebra}  $At_F(E)\subset \ca D(E)$ to be those operators operators whose symbol belongs to $T_F\subset End_{\O_Y}(E)\otimes_{\O_Y}T_F$. By coisotropness $At_F(E)$ is a Lie algebra since $T_F$ is involutive.  We can also define $At_F(E)$ by the following \emph{null foliation Atiyah sequence}  
$$
0\to End_{\O_Y}(E)\to At_F(E)\to T_F\to 0
$$

\begin{theorem}
If $P$ is non-degenerate along $Y$, existence of a first order deformation is equivalent to the existence of a $k$-linear splitting of the null foliation Atiyah sequence which is a $(1/2,1)$-connection.  Furthermore, if $\alpha_2^X$ is symmetric then a second order deformation exists if and only if the splitting agrees with the bracket. 
\end{theorem}

\begin{proof}
The first part is a restatement of Theorem \ref{firstorder}. By definition a splitting, $\gamma$, commutes with the bracket when $c(\gamma)=0$.  Since $\alpha_2^X$ is symmetric this happens if and only if there is a second order deformation.
\end{proof}

\section{Curved DGLA on Hochschild complex}
\subsection{$L_{\infty}$-algebras}
In this section we define \emph{strongly homotopy Lie algebras}, commonly known as \emph{$L_{\infty}$-algebras}.  We give the definitions and results in the curved $L_{\infty}$ case, for lack of a convenient reference.

Let $A$ be a graded vector space over a commutative ring $k$ (not necessarily a field) which contains the rational numbers as a subring.  The homogenous elements of degree $n$ are denoted by $A^n$.  The \emph{suspension} of graded vector space is the graded vector space, $A[1]$, such that $A[1]^n:=A^{n+1}$. Consider the cofree coassociative cocommutative counital coaugmented coalgebra generated by $A[1]$ $S(A[1]):=\oplus_{n\geq 0}S^n(A[1])$ where $S^n(A[1]):=(\otimes^nA[1])^{\Sigma_n}\simeq(\wedge^nA)[n]$ i.e. the set of tensors which are invariant under the natural action of the symmetric group on $n$ elements.  Recall, a counital coalgebra is coaugmented if there exists a coalgebra morphism $\eta:k\to C(V)$.  The notion of a dg-coalgebra morphism will be defined shortly.  The coalgebra structure is given by 
$$
\Delta(a_1\wedge \cdots \wedge a_n)=\sum_{i=1}^n\sum_{\sigma \in \Sigma_{i,n-i}}e(\sigma)(a_{\sigma(1)}\wedge \cdots \wedge a_{\sigma(i)})\otimes (a_{\sigma(i+1)}\wedge \cdots \wedge a_{\sigma(n)})
$$  
where $\Sigma_{i,n-i}$ is the set of $(i,n-i)$-shuffles of $\Sigma_n$ i.e. $\sigma \in \Sigma_n$ such that $\sigma(1)<\cdots <\sigma(i)$ and $\sigma(i+1)<\cdots<\sigma(n)$. The sign is determined by the Koszul rule.

A \emph{curved $L_{\infty}$-algebra structure} on $A$ is a codifferential $Q$ of degree $1$ on $S(A[1])$ i.e. a linear map 
$$
Q:S(A[1])\to S(A[1])[1]
$$
such that $\Delta Q=(Q\otimes id)\Delta+(id\otimes Q)\Delta$ and $Q^2=0$.  In other words $S(A[1])$ has the structure of a dg-coalgebra. Any coderivation is completely determined by the values on the cogenerators given by the composition
$$
\ell_n:S^n(A[1])\to S(A[1])\stackrel{Q}{\to}S(A[1])[1]\to A[2]
$$  
for all $n\geq 0$.  The $\{\ell_k\}_{k\geq 0}$ are known as \emph{higher brackets}.   The condition $Q^2=0$ implies an infinite set of quadratic equations that $\{\ell_n\}_{n\geq 0}$ must satisfy known as \emph{higher Jacobi relations}. If $Q$ agrees with the coaugmentation i.e. $Q\eta =0$ we simply say $A$ is an $L_{\infty}$-algebra.  In the curved case the quadratic relation implies $\ell_1^2=\ell_2\ell_0$ which is nonzero in general hence cohomology is not defined.  Furthermore, for an $L_{\infty}$-algebra $A$ we set $H^*(A):=H^*(A,\ell_1)$.

If $\ell_n=0$ for $n\neq 2$ then $A$ is a graded Lie algebra.  A dg Lie algebra is an $L_{\infty}$ algera with $\ell_n=0$ for $n\neq 1,2$. A \emph{curved dg (CDG) Lie algebra} is an $L_{\infty}$-algebra with $\ell_n=0$ for $n\geq 3$.  The quadratic relation implies $\ell_1\ell_0=0$, $\ell_1\ell_1=\ell_2\ell_0$, $\ell_2$ satisfies the Jacobi identity and $\ell_1$ is a derivation with respect to $\ell_2$. 

\bigskip
\noindent
Let $(\ca C_i,Q_i)$ be two dg-coalgebras a \emph{dg-colagebra morphism} is a morphism of vector spaces $F:\ca C_1\to \ca C_2$ which is equivariant with respect to the codifferentials i.e. $FQ_1=Q_2F$.  In then case when $\ca C_i=S(A_i[1])$ for a graded vector space $A_i$ then such a morphism is determined by a sequence of maps $F_n:\wedge^nA_1\to A_2[1-n]$ for $n\geq 0$ which satisfy an infinite set of equations coming from the compatibility with the codifferentials.  The explicit formulae for a DGLA are in \cite{Kodef}. In this case $F_1$ is a morphism of Lie algebras only up to a homotopy.  In particular the category of dg Lie algebras with dg Lie algebra morphisms is not a full subcategory of the $L_{\infty}$ category.

An $L_{\infty}$-morphism $F:(S(A_1[1]),Q_1)\to (S(A_2[1]),Q_2)$ is a \emph{quasi-isomorphism} if its first component $F_1:A_1\to A_2$ is an isomorphism on cohomology.  Here we are assuming $\ell_0=0$ so cohomology is defined. An important theorem due to Kadeishvili says an $L_{\infty}$-algebra is quasi-isomorphic to its cohomology.

\begin{theorem}\cite{Kadeishvili}   There exists a quasi-isomorphism of $L_{\infty}$-algebras $H^*(A)\to A$ which lifts the identity of $H^*(A)$.
\end{theorem}
\noindent
A dg Lie algebra $A$ \emph{formal} if the induced brackets $\ell_n$ on $H^*(A)$ are $0$ for $n\geq 3$. 
   
\bigskip
\noindent
The zeroes of $Q$ are solutions of the \emph{Maurer-Cartan equation} and put $Zero(Q):=\ca {MC}(A)$.  In terms of the higher brackets $b\in A^1$ is an element of $\ca {MC}(A)$ if and only if 
\begin{equation}
\label{MC}
\sum_{k=0}^{\infty}\frac{1}{k!}\ell_k(b,\ldots,b)=0
\end{equation} 
If $A$ is a dg Lie algebra then \eqref{MC} is the usual Maurer-Cartan equation i.e. $db+\frac{1}{2}[b,b]=0$.  A dg-coalgebra morphism $F:\ca C_1\to \ca C_2$ induces a mapping on solutions of the Maurer-Cartan equation, $F_*:\ca {MC}(\ca C_1)\to \ca {MC}(\ca C_2)$ since it commutes with the codifferentials.

\subsection{Homotopy theory of $L_{\infty}$-algebras} 
The primary difficulty in dealing with curved $L_{\infty}$-algebras is that quasi-isomorphism no longer has any meaning i.e. cohomology is no longer defined.  A replacement is homotopy equivalence which is more general than quasi-isomorphism.  We follow the terminology and exposition of \cite{Fukaya} where the case of $A_{\infty}$-algebras was worked out in detail but little needs to be changed for curved $L_{\infty}$-algebras.  The proofs though are contained in \cite{FOOO}.

In the category of topological spaces the notion of homotopy is very well-known.  In particular, a homotopy between two morphisms $f,f':X\to Y$ is another morphism $H:[0,1]\times X\to Y$ which lives in the category of topological spaces.  Motivated by this there is similar notion of a homotopy in the category of $L_{\infty}$-algebras.  But first we must make sense of the what it means to take the product an $L_{\infty}$-algebra, $A$, with the unit interval.  This is called a \emph{model of $[0,1]\times A$} in \cite{FOOO}.    

\begin{definition}
Define an $L_{\infty}$-algebra $A[1]\otimes k[t,dt]$ where $A[t]$ is the polynomial ring with coefficients in $A$.  An element of $A[1]\otimes k[t,dt]$ is written as a sum $a(t)+b(t)dt$ where $a(t),b(t)\in A[t]$.  Also define $\deg dt=1$ and set $\widetilde{\ell_0}(1)=\ell_0(1)+0dt$.  The higher brackets are given by
\begin{align*}
\widetilde{\ell_1}(a(t)+b(t)dt)&=\ell_1(a(t))-\ell_1(b(t))dt-\frac{db}{dt}dt\\
\widetilde{\ell_k}(a_1(t)+b_1(t)dt,\ldots,a_k(t)+b_k(t)dt)&=\ell_k(a_1(t),\ldots,a_k(t))\\
&+\sum_{j=1}^k(-1)^{|a_1|+\cdots +|a_{j-1}|+j}\ell_k(a_1(t),\ldots,b_j(t),\ldots,a_k(t))dt
\end{align*}

\end{definition} 

We define the $L_{\infty}$ \emph{evaluation homomorphism} $Eval_{t=t_0}:A[1]\otimes k[t,dt]\to A[1]$ as 
$$
Eval_{t=t_0}(a(t)+b(t)dt)=a(t_0)
$$
for $t_0\in \R$.

\begin{definition}
Two $L_{\infty}$ morphisms $f,g:A\to A'$ are \emph{homotopic}, denoted by $f\sim g$, if there exists an $L_{\infty}$ homomorphism $H:A\to A'\otimes k[t,dt]$ such that $Eval_{t=0}\circ H=f$ and $Eval_{t=1}\circ H=g$.
\end{definition}

\begin{definition}
An $L_{\infty}$ morphism $f:A\to A'$ between $L_{\infty}$-algebras is a \emph{homotopy equivalence} if there exists an $L_{\infty}$ morphism $g:A'\to A$ such that $fg\sim id$ and $gf\sim id$.  Furthermore, we say $A$ and $A'$ are \emph{homotopy equivalent} if there exists a homotopy equivalence as above.
\end{definition}

If $\ell_0=0$ then a quasi-isomorphism is the same as a homotopy equivalence.  In the category of $L_{\infty}$-algebras with $L_{\infty}$ homomorphisms a quasi-isomorphism has a homotopy inverse which is also a quasi-isomorphism.  This is not true in the category of dg Lie algebras with dg Lie algebra homomorphisms as there exist quasi-isomorphisms without inverses.  This is one of the reasons to enlarge the dg Lie algebra category to the homotopy $L_{\infty}$-category.  

\bigskip
\noindent
We now introduce a covariant functor $\ca {MC}(A)$ from the category of Artin $k$-local algberas to the category of sets.  Let $\ca R$ be such an algebra, which we consider as a graded algebra concentrated in degree $0$, and $\frak m_{\ca R}$ the maximal ideal.  Since $\ca R$ is concentrated in degree $0$ we have $(A\otimes \frak m_{\ca R})^i=A^i\otimes \frak m_{\ca R}$.  Define the functor $\ca {MC}(A)(\ca R):=\ca {MC}(A\otimes \frak m_{\ca R})$.  If $\psi:\ca R\to \ca R'$ is a morphism of algebras and $b\in \ca {MC}(A\otimes \ca R)$ then $(1\otimes \psi)(b)\in \ca{MC}(A\otimes \frak m_{\ca R'})$.  Hence there is a morphism $\ca {MC}(A)(\ca R)\to \ca {MC}(A)(\ca R')$ which makes $\ca {MC}(A)$ into a covariant functor.  However, the set $\ca {MC}(A)(\ca R)$ is too large to be homotopy invariant so instead we look at equivalence classes in $\ca {MC}(A)(\ca R)$.   

\begin{definition}
Let $b,b'\in \ca {MC}(A)(\ca R)$ then $b$ and $b'$ are \emph{gauge equivalent} denoted by $b\sim b'$ if there exists an element $\widetilde b\in \ca {MC}(A\otimes k[t,dt])(\ca R)$ such that $Eval_{t=0}\circ\widetilde{b}=b_0$, $Eval_{t=1}\circ\widetilde{b}=b_1$.
 \end{definition}
The proof that gauge equivalence is an equivalence relation is found in \cite{FOOO}.  Using this define the \emph{deformation set} as 
$$
Def(A)(\ca R):=\ca {MC}(A)(\ca R)/\sim
$$
For $\psi:\ca R\to \ca R'$ there is a morphism $\psi_*:\ca {MC}(A)(\ca R)\to \ca {MC}(A)(\ca R')$.  Thus there is a \emph{deformation functor} $Def(A)$ from algebras as above to the category of sets.  The following theorem provides justification for taking gauge equivalence classes of solutions to the Maurer-Cartan equation.  

\begin{theorem}\cite[Theorem 2.2.2]{Fukaya}
\label{Defbijection}
If $A$ is homotopy equivalent to $A'$ then the deformation functor $Def(A)$ is equivalent to $Def(A')$.
\end{theorem}

One can extend the above discussion to projective limits of Artin local algebras.  In this paper we will consider the usual projective limit: $\ep k[[\ep]]=\varprojlim(\ep k[\ep]/\ep^rk[\ep])$ cf. \cite{Kodef}.

By definition $\widetilde b=a(t)+b(t)dt\in \ca {MC}(A\otimes k[t,dt])(\ca R)$ if and only if 
 \begin{align*}
 \frac{da}{dt}+\sum_{k=1}^{\infty}\frac{1}{(k-1)!}\ell_k(b(t),a(t),\ldots,a(t))=0\\
 \sum_{k=0}^{\infty}\ell_k(a(t),\ldots,a(t))=0
 \end{align*}
For the first equation we used skew-symmetry of the $\ell_k$. This implies $a(t)\in \ca {MC}(A)(\ca R)$ for all $t$.  If $A$ is a dg Lie algebra gauge equivalence reduces to $da/dt=\ell_1 (b)+\ell_2(b,a)$ cf. \cite{Kodef}.

\bigskip
\noindent
As noted above the difficulty in dealing with curved $L_{\infty}$-algebras is cohomology is not defined.  To overcome this we can \emph{twist} by a Maurer-Cartan element to an $L_{\infty}$-algebra with $\ell_0=0$.  Suppose $b\in \ca {MC}(A)$ and define 
\begin{equation*}
\ell_k^b(a_1,\ldots,a_k)=\sum_{j=0}^{\infty}\frac{1}{j!}\ell_{k+j}(\underbrace{b,\ldots,b}_{j-\text{times}},a_1,\ldots,a_k)
\end{equation*} 
In particular, 
$$
\ell^b_0(1)=\sum_{j=0}^{\infty}\frac{1}{j!}\ell_j(b,\ldots,b)=0
$$
It is then a straightforward calculation to see $(A,\ell_k^b)$ is an $L_{\infty}$-algebra which we call the \emph{$b$-twist of $A$}.

\begin{definition}
The \emph{b-twisted cohomology} of an $L_{\infty}$-algebra $A$ is 
$$
H^*_b(A):=H^*(A,\ell_1^b)
$$
\end{definition}

\begin{proposition}
\cite[Proposition 4.3.16]{FOOO}
\label{gaugeinvariance}
If $b_0\sim b_1$ then there is a homotopy equivalence $f:(A,\ell_k^{b_0})\to (A,\ell_k^{b_1})$.  This implies $b$-twisted cohomology depends only on the gauge equivalence class of $b$.
\end{proposition}

\subsection{Commutative Deformations}

Let $A$ be a ring and $E$ be an $A$-module the Hochschild complex, $\frak g_E^{n}:=Hom_k(A^{\otimes n}\otimes E,E)$, is a dg Lie algebra. The differential is given by 
\begin{align*}
d_{Hoch}\alpha(a_1,\ldots,a_{k+1},e)&:=a_1\alpha(a_2,\ldots,a_{k+1},e)+\sum_{i=1}^{k}(-1)^i\alpha(a_1,\ldots,a_{i}a_{i+1},\ldots,a_{k+1},e)\\
&+(-1)^{k+1}\alpha(a_1,\ldots,a_{k},a_{k+1}e)
\end{align*}
 and the Lie bracket is 
$$
[\alpha_1,\alpha_2]_G:=\alpha_1\circ \alpha_2-(-1)^{kl}\alpha_2\circ \alpha_1
$$
where $\alpha,\alpha_1\in \frak g^k$ and $\alpha_2\in \frak g^l$ and 
\[\alpha_1\circ \alpha_2(a_1,\ldots,a_{k+l},e)=\alpha_1(a_1,\ldots,a_{k},\alpha_2(a_{k+1},\ldots,a_{k+l},e))\]
We will just write $\frak g$ instead of $\frak g_E$ when there is no confusion.  It is the well known in this case that \cite[Lemma 9.1.9]{Weibel}
$$
H^*(\frak g)=Ext_A^{*}(E,E)
$$

In this section we do not assume $Y$ is coisotropic. Let $X$ be an affine scheme and $Y$ a subvariety with a vector bundle $E$. A \emph{commutative deformation} of $E$ as a coherent $\O_X$-module is a flat deformation to an $\O_X$-module.  The module structure is given by 
\begin{equation}
\label{starmodule}
a\star e=ae+\ep \alpha_1(a,e)+\ep^2\alpha_2(a,e)+\cdots
\end{equation}
If we define $\alpha^E:=\ep\alpha_1+\ep^2\alpha_2+\cdots\in \frak g^1[[\ep]]$ then associativity of \eqref{starmodule} is equivalent to the Maurer-Cartan equation in $\frak g[[\ep]]$
\begin{equation}
\label{MCcom}
d_{Hoch}\alpha^E+\frac{1}{2}[\alpha^E,\alpha^E]_G=a\alpha^E(b,e)-\alpha^E(ab,e)+\alpha^E(a,be)+\alpha^E(a,\alpha^E(b,e))
\end{equation}
Two solutions $\alpha^E,(\alpha^E)'$ are \emph{gauge equivalent} denoted by $\alpha^E\sim (\alpha^E)'$ if there exists a $\phi\in \frak g^0[[\ep]]$ such that 
\begin{equation}
\label{gauge}
\phi(a\star e)=a\star'\phi(e)
\end{equation}
which restricts to the identity modulo $\ep$. Such a $\phi$ is of the form $\phi=id+\ep\phi_1+\ep^2\phi_2+\cdots$ and \eqref{gauge} is equivalent to 
\begin{equation}
\label{gaugen}
\phi_n(ae)+\alpha_n(a,e)+\sum_{j+k=n-1}\phi_{j}(\alpha_k(a,e))=a\phi_n(e)+\alpha_n'(a,e)+\sum_{j+k=n-1}\alpha_j'(a,\phi_k(e))
\end{equation} 
for all $n\geq 1$.  It is then straightforward to check that gauge equivalence as defined above is equivalent to gauge equivalence defined in the previous subsection cf. \cite[Section 3.2]{Kodef}.  The main result of this section is the following formality theorem in this setting cf. \cite[Conjecture 2.36]{AC}: 

\begin{theorem}
\label{commutativedef}
In the above setting the dg Lie algebra $(\frak g,d_{Hoch},[,])$ is quasi-isomorphic to the abelian Lie algebra $(\wedge^*N(E),0)$ i.e. $\frak g_E$ is formal.   
\end{theorem}
\noindent
First we need a lemma to compute the cohomology of $\frak g_E$.  

\begin{lemma} 
Let $X$ be a smooth affine variety and $Y$ a subvariety with a vector bundle $E$.  Then there is an isomorphism
$$
H^*(\frak g)\simeq \wedge^*N(E)
$$
\end{lemma}

\begin{proof} By \cite[Lemma 9.1.9]{Weibel} 
\[\ca H^p(\frak g)=\ca Ext^p_{\O_X}(E,E)\]

The rest of the lemma is a special case of a more general calculation found in \cite{CKS}. They compute $Ext_{\O_X}(E_1,E_2)$ where $E_i$ are vector bundles supported on possibly distinct subvarieties.  In our case the calculation simplifies dramatically so we include it for completeness.

 To calculate $\ca Ext^p_{\O_X}(E,E)$ for $E$ a vector bundle we use the change of ring spectral sequence:
\[\iota_*\ca Ext^p_{\O_Y}(E,\ca Ext^q_{\O_X}(\O_Y,E))\Rightarrow \ca Ext^{p+q}_{\O_X}(E,E)\]
Since $X$ is affine and $E$ is locally free over $Y$ the spectral sequence becomes 
\[\ca Ext^p_{\O_X}(E,E)=\ca Hom_{\O_Y}(E,\ca Ext^p_{\O_X}(\O_Y,E))=E^{*}\otimes_{\O_Y}\ca Ext^p_{\O_X}(\O_Y,E)\]
A lemma is needed in order to calculate $\ca Ext^p_{\O_X}(\O_Y,E)$
\begin{lemma} $\ca H_k\L\iota^*\iota_* \O_Y=\wedge^kN_{Y/X}^*$ \end{lemma}  
\begin{proof} 
First notice

\begin{equation}
\label{projection}
\iota_*\L\iota^*\iota_* \O_Y= \iota_*(\L \iota^*\iota_*\O_Y\lotimes_{\O_Y}\O_Y )=\iota_*\O_Y\lotimes_{\O_X} \iota_*\O_Y 
\end{equation}
where the second equality is the projection formula.  Since $\iota$ is a closed embedding the underived pullback of the direct image fixes the sheaf.  This gives
\begin{align*} \ca H_k\L\iota^*\iota_* \O_Y&=\iota^*\iota_*\ca H_k\L\iota^*\iota_*\O_Y\\
&=\iota^*\ca H_k\iota_*\L\iota^*\iota_*\O_Y\\
&=\iota^*\ca H_k(\iota_*\O_Y\lotimes_{\O_X} \iota_* \O_Y)\\
&=\wedge^kN_{Y/X}^* \end{align*} 
The first equality is the above remark about $\iota$ being a closed embedding, the second uses $\iota_*$ is an exact functor, third is \eqref{projection} from above.  The last equality uses the well known fact that $\iota^*Tor_k^{\O_X}(\O_Y,\O_Y)=\wedge^kN^*_{Y/X}$. \end{proof}

Returning to the original calculation 
\begin{align*}\ca Ext^k_{\O_X}(\iota_*\O_Y,\iota_*E)&=\ca H^k\textbf{R}\ca Hom_{\O_X}(\iota_*\O_Y,\iota_*E)\\
&=\ca H^k\iota_*\textbf{R}\ca Hom_{\O_Y}(\L\iota^*\iota_* \O_Y,E)\\
&=\iota_*\ca H^k\textbf{R}\ca Hom_{\O_Y}(\L\iota^*\iota_*\O_Y,E)\end{align*}
The second equality is the adjoint relation $\L\iota^*\dashv \iota_*$ \cite[2.5.10]{RD} and the last equality is exactness of $\iota_*$.  The Grothendieck spectral sequence in this case yields
\[\ca H^p\textbf{R}\ca Hom_{\O_Y}(\ca H_q\L\iota^*\iota_*\O_Y,E)\Rightarrow \ca H^{p+q}\textbf{R}\ca Hom_{\O_Y}(\L\iota^*\iota_*\O_Y,E)\]
By the above
\[\iota_*\ca H^p\textbf{R}\ca Hom_{\O_Y}(\ca H_q\L\iota^*\iota_*\O_Y,E)=\iota_*\ca H^p\textbf{R}\ca Hom_{\O_Y}(\wedge^qN^*_{Y/X},E)=\iota_*\ca H^p(Y,\wedge^qN_{Y/X}\otimes_{\O_Y}E)\]
This implies since $Y$ has no higher cohomology that
\[\ca Ext_{\O_X}^k(\iota_*\O_Y,\iota_*E)=\iota_*\ca H^0(Y,\wedge^kN_{Y/X}\otimes_{\O_Y} E)=\iota_*(\wedge^kN_{Y/X}\otimes_{\O_Y} E)\]

 \end{proof}

\begin{proof}[Proof of Theorem 4.9]
First construct a contraction from $\frak g\stackrel{\pi}{\to}\wedge^*N(E)$  which exists since the ground ring contains the rational numbers as a subring.  By comparing symmetry properties the map
$$
\frak g\otimes\frak g\stackrel{[\cdot,\cdot]_G}{\to}\frak g\stackrel{\pi}{\to} \wedge^*N(E)
$$
is identically $0$. Now use the arguments from \cite[Section 2]{HuebStas} to see that $\wedge^*N(E)$ has no higher brackets. This is the definition that $\frak g$ is formal.   
\end{proof}

Applying Theorem \ref{Defbijection} and using that a homotopy equivalence is the same as a quasi-isomorphism we get the following corollary
\begin{corollary}
There is a bijection between commutative deformations up to equivalence and section of $N(E)$.
\end{corollary}

\subsection{Noncommutative Deformations}

Once Poisson structures are introduced the problem is far more elaborate. Let $X$ be an affine Poisson variety with Poisson bivector $P\in \Gamma(X,\wedge^2T_X)$.  A deformation quantization of $\O_X$ is an associative product of the form
$$
a\star b=ab+\ep\alpha_1^X(a,b)+\ep^2\alpha_2^X(a,b)+\cdots
$$ 
where $a,b\in \O_X$ and $\alpha_i^X$ are bi-differential operators.  Associativity of $\star$ is equivalent to 
\begin{equation}
\label{assX}
a\alpha^X(b,c)-\alpha^X(ab,c)+\alpha^X(a,bc)-\alpha^X(a,b)c-\alpha^X(\alpha^X(a,b),c)+\alpha^X(a,\alpha^X(b,c))=0
\end{equation}
where $\alpha^X:=\ep\alpha_1^X+\ep^2\alpha_2^X+\cdots$ and $a,b,c\in \O_X$.  Equation \eqref{assX} is the Maurer-Cartan equation in the Hochschild complex of $\O_X$ with the usual Hoschschild differential and Gerstenhaber bracket \cite{Kodef}.

 Given a subvariety $Y\subset X$ and a vector bundle $E$ on $Y$ which we view as a coherent $\O_X$-module define a quantization of $E$ as a flat coherent $\ca A$-module, $\ca E$.  Immediately from the above $Y$ must coisotropic.  The module action is still given by \eqref{starmodule} but associativity of the action is 
$$
a\alpha^E(b,e)-\alpha^E(ab,e)+\alpha^E(a,be)-\alpha^X(a,b)e-\alpha^E(\alpha^X(a,b),e)+\alpha^E(a,\alpha^E(b,e))=0
$$
We define gauge equivalence as in \eqref{gauge} which is still equivalent to \eqref{gaugen} for all $n\geq 1$.

Unlike the commutative case, noncommutative deformations are not governed by a dg Lie algebra but a curved dg Lie algebra.  Define 
\begin{enumerate}
\item[(1)]$\ell_0(1):=-\alpha_X\otimes id_E$
\item[(2)]
$\ell_1(\alpha)(a_1,\ldots,a_{k+1},e):=d_{Hoch}\alpha(a_1,\ldots,a_{k+1},e)+\sum_{j=1}^{k}(-1)^j\alpha(a_1,\ldots,\alpha_X(a_j,a_{j+1}),a_{j+2},\ldots,a_{k+1},e)$
\item[(3)]$\ell_2(\alpha_1,\alpha_2):=[\alpha_1,\alpha_2]_G$
\end{enumerate}
The fact these make $\frak g$ into a curved dg Lie algebra is a straightforward computation.  In particular, $\ell_1(\ell_0(1))=0$ is precisely the Maurer-Cartan equation in $C^*(\O_X,\O_X)[[\ep]]$.  Associativity of \eqref{starmodule} is then given by solutions of the Maurer-Cartan equation 
\begin{equation}
\label{Maurer-Cartan}
\ell_0(1)+\ell_1(\alpha)+\frac{1}{2}\ell_2(\alpha,\alpha)=0
\end{equation}
That gauge equivalence is equivalent to that defined in section 4.2 follows since $\ell_1(\beta)=d_{Hoch}(\beta)$ for $\beta\in \frak g^0$.  Hence the deformation space controls noncommutative module deformations up to gauge equivalence.

Given a solution, $\alpha$, of \eqref{Maurer-Cartan} define a deformed derivation by $\ell^{\alpha}_1(\beta):=\ell_1(\beta)+\ell_2(\alpha,\beta)$.  It is easy to check that $\ell_1^\alpha\ell^\alpha_1=0$ is equivalent to \eqref{Maurer-Cartan} so $\ell^{\alpha}_1$ defines a differential on $\frak g[[\epsilon]]$.  There is a deformed dg Lie algebra structure on $\frak g[[\ep]]$ given by $(\ell_1^\alpha,\ell_2)$.

\begin{definition}
Let $E$ be a vector bundle on $Y$ which has a deformation quantization, $\alpha^E$.  The \emph{Poisson-Hochschild cohomology of $E$} is defined as 
\begin{equation}
\label{quantizedcohomology}
HP^*(Y,E,\alpha^E):=H^*(\frak g[[\ep]],\ell_1^{\alpha^E})
\end{equation}
\end{definition}

When $P$ is nondegenerate and $Y$ is Lagrangian with a line bundle $L$ the ``classical'' limit of $HP^*(Y,L,\alpha^L)$ recovers the de Rham cohomology of $Y$.  In a subsequent paper we will discuss the construction of a category consisting of pairs $(Y,E)$ where $Y$ is a coisotropic subvariety and $E$ is a vector bundle supported on $Y$ which has a deformation quantization.  The endomorphisms will be the Poisson-Hochschild cohomology of $E$ cf. \cite{CaFed, KapOr, Soibelman}.

\appendix
\section{}

\subsection{Local equations for deformations.}

Here we collect the standard formulas describing a second order
deformation $\mathcal{A}_2$. The first two formulas must hold on
each open  subset $U_i$ of an affine covering. To unload notation
we write $\alpha_2^X$ instead of 
 $\alpha_2^{Xi}$.
\begin{equation}
\label{alphaX-one}
a \alpha_1^X(b, c) - \alpha_1^X(ab, c) + 
\alpha_1^X(a, bc) - \alpha_1^X(a, b) c  =  0 \hspace{5.8cm}
\end{equation} 
\begin{equation}
\label{alphaX-two}
a \alpha_2^X(b, c) - \alpha_2^X(ab, c) + 
\alpha_2^X(a, bc) - \alpha_2^X(a, b) c  =  
\alpha_1^X(\alpha_1^X(a, b), c) - \alpha_1^X(a, \alpha_1^X(b, c))
\end{equation} 
The next two formulas must hold on each double intersection $U_i \cap U_j$;
we write $\beta_1^X$ and $\beta_2^X$ instead of
$\beta_1^{Xij}$ and $\beta_2^{Xij}$, respectively. 
\begin{equation}
\label{betaX-one}
\beta_1^X(ab)- a \beta_1^X(b) - \beta_1^X(a)b  =  0 \hspace{9cm}
\end{equation} 
\begin{equation}
\label{betaX-two}
\beta_2^X(ab)- a \beta_2^X(b) - \beta_2^X(a)b  =   \hspace{9.3cm}
\end{equation} 
$$
\hspace{1cm} =  \alpha_2^{Xj}(a, b) - \alpha_2^{Xi}(a, b) +
\beta_1^X(a) \beta_1^X(b) - \beta_1^X(\alpha_1^X(a, b))
+ \alpha_1^X(\beta_1^X(a), b) + \alpha_1^X(a, \beta_1^X(b))  
$$
We also give similar equations for the module action: 
\begin{equation}
\label{alpha-one}
a \alpha_1(b, e) - \alpha_1(ab, e)+ \alpha_1(a, be)  = \alpha_1^X(a, b) e
\hspace{7cm}
\end{equation}
\begin{equation}
\label{alpha-two}
a \alpha_2(b, e) - \alpha_2(ab, e)+ \alpha_2(a, be)  = \alpha_2^X(a, b) e
+ \alpha_1(\alpha_1^X(a, b), e) - \alpha_1(a, \alpha_1(b, e))
\hspace{1cm}
\end{equation}
\begin{equation}
\label{beta-one}
\beta_1(ae)- a \beta_1(e) = \alpha_1^j (a, e) - \alpha_1^i (a, e)
+ \beta_1^X(a)e
\hspace{6cm}
\end{equation}
\begin{equation}
\label{beta-two}
\beta_2(ae)- a \beta_2(e) = 
\hspace{11cm}
\end{equation}
$$
\hspace{1cm}=
\alpha_2^j (a, e) - \alpha_2^i (a, e)
+ \beta_2^X(a)e + \alpha_1^j(a, \beta_1(e)) - \beta_1(\alpha_1^i(a, e)))
+ \beta_1^X(a) \beta_1(e) + \alpha^j_1(\beta_1^X(a), e) 
$$
In addition, there should be equalities on the triple intersections
$U_i \cap U_j \cap U_k$ saying that the transition functions satisfy
the cocycle condition. Only the module version of these equations
is relevant to this paper:
$$
\beta_1^{kj} +\beta_1^{ji} - \beta_1^{ki} = 0; \quad
\beta_2^{kj} +\beta_2^{ji} - \beta_2^{ki} = \beta_1^{kj} \circ \beta_1^{ji}
$$
However, we will avoid dealing with these equations directly
by assuming that $H^2(Y, \O_Y(E))=0$. In our applications
the difference $LHS-RHS$ is always $\O_Y$-linear and 
satisfies the cocycle condition on the fourfold intersections.
Since the 2-cocycle can always be resolved due to the assumption,
we can adjust $\beta_*^{ji}$ to ensure that the last two 
equations hold as well. 

\begin{lemma}
\label{HKR}
Let $A$ be the ring of regular functions on a smooth affine variety $X$ and $E$ a projective module of finite rank over the quotient ring $B$ corresponding to a smooth affine subvariety $Y$.  Let $R:A\otimes_k E\to E$ be a $k$-linear map. Then $\beta(ae)-a\beta(e)=R(a,e)$ for some $\beta\in Hom_k(E,E)$ if and only if $R$ vanishes in $I\otimes_k E$ and also satisfies
$$
aR(b,e)-R(ab,e)+R(a,be)=0
$$
Similarly, if $G:A\otimes_kA\otimes_k E\to E$ is a $k$-linear map then $a\rho(b,e)-\rho(ab,e)+\rho(a,be)=G(a,b,e)$ for some $\rho:A\otimes_kE\to E$ if and only if the restriction of $G$ to $I\otimes_kI\otimes_kE\to E$ is symmetric in the first two arguments and 
$$
aG(b,c,e)-G(ab,c,e)+G(a,bc,e)-G(a,b,ce)=0
$$
Moreover, if $R$, resp. G is an algebraic differential operator in each of its arguments then one can choose $\beta$, resp. $\rho$, with the same property.
\end{lemma}

\begin{lemma}
\label{anticocycle}
Let $\beta\in \frak g^i_E$ (see section 4.3 for the notation) be a cocycle for $i=2,3$ then the antisymmetrization of $\beta$ when restricted to $I_Y$ is $\O_X$ polylinear.
\end{lemma}

\noindent
\begin{proof}
The conclusion for $i=2$ is clear.  For $i=3$ we have 
\[a\beta(b,c,e)-\beta(ab,c,e)+\beta(a,bc,e)-\beta(a,b,ce)=0\]
for all $a,b,c\in A$.  Hence for $a\in A, x,y \in I$
\begin{align*} a\beta(x,y,e)-a\beta(y,x,e)-\beta(ax,y,e)+\beta(y,ax,e)&=-\beta(a,xy,e)-a\beta(y,x,e)+\beta(ay,x,e)\\
&=-\beta(a,xy,e)+\beta(a,yx,e)=0
 \end{align*}
and 
\begin{align*}a\beta(x,y,e)&-a\beta(y,x,e)-\beta(x,y,ae)+\beta(y,x,ae)\\
&=a\beta(x,y,e)-a\beta(y,x,e)+\beta(xy,a,e)-\beta(x,ya,e)+\beta(y,x,ae)\\
&=a\beta(x,y,e)-a\beta(y,x,e)+\beta(xy,a,e)-\beta(x,ya,e)-\beta(yx,a,e)+\beta(y,xa,e)\\
&=a\beta(x,y,e)-a\beta(y,x,e)-\beta(ax,y,e)+\beta(ay,x,e)\\
&=-\beta(a,xy,e)+\beta(a,yx,e)=0 
\end{align*}
\end{proof}


\end{document}